\documentclass[10pt]{amsart}
\usepackage{graphicx}
\usepackage{amscd}
\usepackage{amsmath}
\usepackage{amsthm}
\usepackage{amsfonts}
\usepackage{amssymb}
\usepackage{mathrsfs}
\usepackage{bm}
\usepackage{enumerate}
\usepackage{amsrefs}
\usepackage{xcolor}
\usepackage[colorlinks, citecolor=blue, linkcolor=red, pdfstartview=FitB]{hyperref}
\usepackage{comment}

\usepackage{tikz}
\usepackage{tikz-cd}
\usepackage{caption}
\usetikzlibrary{decorations.markings}
\usepackage{lipsum}
\usepackage{mathrsfs}
\usepackage{enumitem}

\usepackage{marginnote}	% for editing
\usepackage{soul} % for indicating new material

\newcommand{\bB}{{\mathbb{B}}}

\newcommand{\bN}{{\mathbb{N}}}

\newcommand{\bT}{{\mathbb{T}}}

\newcommand{\bZ}{{\mathbb{Z}}}

  \newcommand{\A}{{\mathcal{A}}}
  \newcommand{\B}{{\mathcal{B}}}
  \newcommand{\C}{{\mathcal{C}}}

  \newcommand{\F}{{\mathcal{F}}}
  
\renewcommand{\H}{{\mathcal{H}}}
  \newcommand{\I}{{\mathcal{I}}}
  
  \newcommand{\K}{{\mathcal{K}}}  
\renewcommand{\L}{{\mathcal{L}}}

\renewcommand{\O}{{\mathcal{O}}}

  \newcommand{\T}{{\mathcal{T}}}

\newcommand{\rC}{\mathrm{C}}

\renewcommand{\phi}{\varphi}

\newcommand{\upchi}{{\raise.35ex\hbox{$\chi$}}}

\newcommand{\ol}{\overline}

\newcommand{\Ad}{\operatorname{Ad}}

\newcommand{\mycomment}[1]{}

\newcommand{\id}{\operatorname{id}}

\newcommand{\spn}{\operatorname{span}}

\newcommand{\acts}{\curvearrowright}

\newtheorem{theorem}{Theorem}[section]
\newtheorem*{theorem*}{Theorem}

\newtheorem*{lemma*}{Lemma}

\newtheorem{proposition}[theorem]{Proposition}
\newtheorem{definition}[theorem]{Definition}

\newtheorem{theoremx}{Theorem}

\date{\today}

\author{Adam Dor-On}
\address{Department of Mathematics, University of Haifa, Mount Carmel, Haifa 3103301, Israel}
\email{adoron.math@gmail.com\vspace{-1ex}}

\author{Ian Thompson}
\address{Department of Mathematical Sciences, University of Copenhagen, Universitetsparken 5, 2100 Copenhagen, Denmark}
\email{ian@math.ku.dk\vspace{-1ex}}

\subjclass[2020]{46K50, 47L55, 46L55, 46L05.}

\keywords{Hao-Ng isomorphism, C*-envelope, noncommutative dynamics, reduced crossed product,  unique extension, Cuntz-Pimsner algebra.}

\thanks{A. Dor-On was partially
supported by an NSF-BSF grant no. 2350543 / 2023695 (respectively), a Horizon Marie-Curie SE project no.
101086394 and a DFG Middle-Eastern collaboration project no. 529300231. I. Thompson was partially supported by NSERC Postdoctoral Fellowship no. 587750.}

\title[The Hao-Ng isomorphism theorem for reduced crossed products]{The Hao-Ng isomorphism theorem \\ for reduced crossed products}
\begin{document}
\begin{abstract}
    We prove the Hao-Ng isomorphism for reduced crossed products by locally compact Hausdorff groups. More precisely, for a non-degenerate $\rC^*$-correspondence $X$ and a generalized gauge action $G \acts X$ by a locally compact Hausdorff group $G$, we prove the commutation $\O_{X\rtimes_rG}\cong\O_X\rtimes_rG$ of the reduced crossed product with the Cuntz-Pimsner $\rC^*$-algebra construction. Our proof shows how such commutations with reduced crossed products can follow from more general principles in non-self-adjoint crossed product theory.
\end{abstract}
\maketitle

\section{Introduction}\label{s:intro}

Alongside the emergence of category theory as a unifying language for abstracting various constructions across Mathematics, determining the precise relationship between two given functorial operations has become crucial for advancing far-reaching structural theories. This is especially true in operator $K$-theory and, in particular, Kasparov's $KK$-theory and Higson's $E$-theory. In the setting of $KK$-theory, Schochet's K\"unneth formula \cite{schochet1982topological} is a prime example of this where the relationship between $K$-theory and tensor products is determined, and has led to tremendous impact in the classification program for $\rC^*$-algebras \cite{rosenberg1987kunneth, baum1994classifying}.

In this paper, we resolve the reduced Hao-Ng isomorphism problem in complete generality, establishing the commutation of the reduced crossed product functor with the Cuntz-Pimsner $\rC^*$-algebra functor. We briefly describe the problem at hand, where more details can be found in Section \ref{s:prelim}. Given a (non-degenerate) $\rC^*$-correspondence $X$ over a $\rC^*$-algebra $\B$, the smallest gauge-equivariant $\rC^*$-algebra $\O_X$ generated by $X$ and $\B$ is called the Cuntz-Pimsner algebra of $X$. The Cuntz-Pimsner construction can be largely understood at the level of the underlying $\rC^*$-correspondence \cites{katsura2007ideal, katsura2004c}, yet it encompasses a wide variety of naturally occurring $\rC^*$-algebras, including all crossed products by $\bZ$ and topological graph algebras \cites{katsura2004c, muhly2005topological, pimsner1996class}. Given a locally compact Hausdorff group $G$, a generalized gauge action on $X$ is an action $\alpha:G\acts\O_X$ that preserves the copies of $X$ and $\B$ inside $\O_X$, and allows us to form the reduced crossed product $\rC^*$-correspondence $X\rtimes_{r,\alpha} G$ in a natural way. The reduced Hao-Ng isomorphism then asks whether the Cuntz-Pimsner algebra $\O_{X\rtimes_{r,\alpha}G}$ of $X\rtimes_{r,\alpha}G$ is isomorphic to the corresponding reduced crossed product $\O_X \rtimes_{r,\alpha} G$ of the Cuntz-Pimsner algebra of $X$.

The Hao-Ng isomorphism problem was first considered around 18 years ago by Hao and Ng \cite{hao2008crossed} where they established the validity of the isomorphism described above for actions by amenable locally compact Hausdorff groups. In the decade since then, the Hao-Ng isomorphism problems for both full and reduced crossed products were shown to be intimately tied to various functoriality and Takai-type duality for crossed products \cites{abadie2010takai, kaliszewski2013functoriality, kaliszewski2015coactions} and, in the work of B\'edos, Kaliszewski, Quigg, and Robertson \cite{bedos2015new}, the Hao-Ng isomorphism for reduced crossed products was established for actions of discrete exact groups. Further applications can be found in other works, including those of Schafhauser on AF-embeddability \cite{schafhauser2015cuntz}, and of Deaconu on group actions on graph $\rC^*$-algebras \cites{deaconu2012group, deaconu2018group}.

In recent years, significant progress has been made on the Hao-Ng isomorphism problems by establishing a bridgehead between the structure theories of $\rC^*$-algebras and non-self-adjoint operator algebra theory \cites{katsoulis2019crossed, katsoulis2021non}. In these papers, Katsoulis and Ramsey show that to prove the Hao-Ng isomorphism for the reduced crossed product, it is sufficient to prove that the reduced \emph{non-self-adjoint} operator algebra crossed product commutes with the $\rC^*$-envelope. This strategy was highly successful, and has led to the resolution of the Hao-Ng isomorphism problem for reduced crossed products by discrete groups \cites{katsoulis2017c}, as well as under the assumption that the underlying $\rC^*$-correspondence is hyperrigid \cite{katsoulis2021non}. Here, we provide a complete answer by implementing the strategy of Katsoulis and Ramsey.

\begin{theoremx}\label{t:a}
    Let $X$ be a C*-correspondence, $G$ be a locally compact Hausdorff group, and $\alpha:G\acts X$ be a generalized gauge action. Then, $\O_X\rtimes_{\alpha, r} G\cong \O_{X\rtimes_{\alpha, r} G}$.
\end{theoremx}

Over the last several decades the $\rC^*$-envelope has proven itself as a robust tool in operator theory and operator algebras \cite{arveson1969subalgebras, dritschel2005boundary, arveson2008noncommutative, davidson2015choquet}. It was first shown to exist in Hamana's work on injective envelopes for operator systems \cite{hamana1979injective}, leading to far-reaching consequences in group theory \cites{kalantar2017boundaries, breuillard2017c}, noncommutative convexity theory \cites{davidson2017d, davidson2022strongly, davidson2019noncommutative, kennedy2021noncommutative}, and approximation theory \cites{bilich2024arveson, kennedy2015essential, clouatre2024rigidity}.

The idea of establishing commutation results for the $\rC^*$-envelope was also used by the first-named author together with Geffen and Eilers in \cite{dor2020classification}, where connections between the seemingly distinct classification theories for $\rC^*$-algebras and non-self-adjoint operator algebras were established. These types of links between the self-adjoint and non-self-adjoint theories have led to significant advances in the structure theory of $\rC^*$-algebras, including gauge-coaction co-universality theorems for a variety of $\rC^*$-algebras \cites{brix2026normal, dor2020tensor, dor2022c, kakariadis2023couniversality, sehnem2022c}, as well as various generalizations of Hao-Ng isomorphism theorems in the context of product systems \cites{dor2020tensor, dor2022c, katsoulis2020product, kakariadis2025fock,li2022zappa}. In this paper, we establish the commutation of the reduced crossed product functor with the $\rC^*$-envelope under a separation condition \eqref{eq:condition-S}, presented in Section \ref{s:cross-prod-env}.

\begin{theoremx}\label{t:b}
Let $\A$ be an operator algebra with a contractive approximate identity, and let $G$ be a locally compact Hausdorff group. If $\alpha:G\acts\A$ is an action satisfying condition \eqref{eq:condition-S}, then $\rC^*_e(\A)\rtimes_{\alpha,r}G\cong\rC^*_e(\A\rtimes_{\alpha,r}G)$ via the canonical map.
\end{theoremx}

The proof of Theorem \ref{t:b} is inspired by \cite{katsoulis2021non}, where it was established that the reduced crossed product functor commutes with the $\rC^*$-envelope for \emph{hyperrigid} operator algebras (see \cite[Definition 1.1]{arveson2011noncommutative}). One problem in extending the proof strategy of \cite{katsoulis2021non} beyond the hyperrigid setting is rooted in the fact that the unique extension property from noncommutative boundary theory \cite[Definition 2.1]{arveson2008noncommutative} is not preserved under direct integrals. This was first witnessed in a recent counterexample to Arveson's hyperrigidity conjecture by Bilich and the first-named author \cite{bilich2024arveson}. Our approach overcomes this by isolating the more readily verifiable condition \eqref{eq:condition-S}, from which Theorem \ref{t:a} follows from the more general principle of commutation of reduced crossed products with the $\rC^*$-envelope. The question of the general validity of condition \eqref{eq:condition-S}, as well as the unrestricted commutation of the $\rC^*$-envelope with reduced crossed products will be taken up in subsequent work. Finally, we refer the reader to \cite[Section 6]{mukohara2026stabilization}, for an alternative proof of the reduced Hao-Ng isomorphism theorem which applies to degenerate $\rC^*$-correspondences.

This paper has three sections, including this introduction. In Section \ref{s:prelim}, we gather preliminary facts on the $\rC^*$-envelope, the unique extension property, $\rC^*$-correspondences, and Cuntz-Pimsner algebras. In Section \ref{s:cross-prod-env}, we prove Theorem \ref{t:b} and show that condition \eqref{eq:condition-S} holds for dynamical systems arising from generalized gauge actions. This then allows us to establish Theorem \ref{t:a}.

\section{Preliminaries}\label{s:prelim}

\subsection{The C*-envelope and the unique extension property} \label{ss:C*-env}

We recall some of the necessary machinery from operator algebra theory, which may be found in \cites{arveson2008noncommutative, paulsen2002completely}. For potentially non-unital operator algebras we refer to \cite{meyer2001adjoining}, \cite[Subsection 4.3]{blecher2004operator} and \cite[Subsection 2.2]{dor2018full}. In these works, it is explained how various aspects of Arveson's noncommutative boundary theory carries over to the possibly non-unital setting with minor modifications.

An \emph{operator algebra} is a norm-closed subalgebra of bounded operators on a Hilbert space $\A\subseteq \bB(\H)$. A \emph{representation} of $\A$ is a completely contractive homomorphism $\rho:\A\rightarrow\bB(\K)$. The \emph{$\rC^*$-envelope} of $\A$ is a pair $(\rC^*_e(\A), \varepsilon)$ consisting of a $\rC^*$-algebra $\rC^*_e(\A)$ together with a completely isometric representation $\varepsilon:\A\rightarrow\rC^*_e(\A)$ such that $\rC^*(\varepsilon(\A)) = \rC^*_e(\A)$ has the following co-universal property: whenever $\iota:\A\rightarrow\B$ is a completely isometric representation with $\B = \rC^*(\iota(\A))$, then there is a surjective $*$-homomorphism $\pi:\B\rightarrow\rC^*_e(\A)$ such that $\pi\circ\iota = \varepsilon$. Since this property uniquely determines the $\rC^*$-envelope up to a $*$-isomorphism that preserves an image of $\A$, we frequently refer to the $\rC^*$-algebra $\rC^*_e(\A)$ as \emph{the} $\rC^*$-envelope of $\A$.

An important method for constructing the $\rC^*$-envelope of an operator algebra $\A\subseteq \bB(\H)$ is through representations with a particular unique extension property, which is analogous to classical ideas from Choquet theory. For this, suppose that $\B:=\rC^*(\A)$ is the $\rC^*$-algebra generated by $\A$ in some completely isometric embedding of $\A$ in $\bB(\H)$. A $*$-representation $\pi:\B\rightarrow\bB(\K)$ is said to have the \emph{unique extension property} with respect to $\A$ if there is a unique completely positive completely contractive map $\psi:\B\rightarrow\bB(\K)$ such that $\psi|_\A = \pi|_\A$. Now, if $\pi:\B\rightarrow \bB(\K)$ has the unique extension property with respect to $\A$ and $\pi|_\A$ is completely isometric, then it is easy to show that $(\pi(\B), \pi|_\A)$ coincides with the $\rC^*$-envelope of $\A$ (see \cite[Subsection 2.2]{dor2018full}). In the unital setting, the existence of such $*$-representations was first exhibited by Dritschel and McCullough \cite{dritschel2005boundary}, but straightforward arguments as in \cite[Subsection 2.2]{dor2018full} show that this carries over to the possibly non-unital setting.

\subsection{C*-correspondences, Cuntz-Pimsner algebras, and their dynamics}\label{ss:correspondence}

Here, we record some of the details on $\rC^*$-correspondences that we shall need in this paper, and refer the reader to \cite[Section 4.6]{brown2008textrm} or \cite{lance1995hilbert} for more. Let $\B$ be a $\rC^*$-algebra. For a right Hilbert $\B$-module $X$, we denote by $\L(X)$ the $\rC^*$-algebra of adjointable operators and by $\K(X)$ its subalgebra given by the closure of the ``rank-one" operators. More precisely, for $x,y\in X$, the rank-one operators are denoted by $\theta_{x,y}$ and defined by $\theta_{x,y}(z) := x\langle y, z\rangle$.

Recall that a \emph{$\rC^*$-correspondence} over $\B$ is a right Hilbert $\B$-module $X$ together with a $*$-homomorphism $\varphi_X: \B\rightarrow\L(X)$, which we consider as a left action of $\B$ on $X$. We will assume throughout this paper that $\rC^*$-correspondences are \emph{non-degenerate}, where we say that $X$ is non-degenerate if the $\varphi_X(\B)X$ is dense in $X$. 

Given two $\rC^*$-correspondences $X$ and $Y$, a new $\rC^*$-correspondence $X\otimes Y$ may be formed in the following way: First, on the quotient of the algebraic tensor product $X\odot Y$ by \[xb\otimes y - x\otimes \varphi_Y(b)y, \quad x\in X, y\in Y, b\in \B,\]
we may define a $\B$-valued sesquilinear form and a two-sided $\B$-action by setting \[\langle x\otimes y, v\otimes w\rangle = \langle y, \varphi_Y(\langle x, v\rangle)w\rangle,\] 
\[ (x\otimes y)b = x\otimes(yb), \quad \text{and} \quad \varphi_{X\otimes Y}(b)(x\otimes y) = (\varphi_X(b)x)\otimes y\]
for every $v,x\in X, w,y\in Y$ and $b\in\B$. Then, the completion with respect to the $\B$-valued inner product yields a $\rC^*$-correspondence $X\otimes Y$.

For a $\rC^*$-correspondence $X$, one may then construct the Fock correspondence
\[\F_X = \B \oplus \bigoplus_{n=1}^{\infty} X^{\otimes n}\]
over $\B$. The Fock correspondence gives rise to the so-called left-creation operators $T_x\in\L(\F_X)$ for $x\in X,$ defined by
\[T_x(b) = xb \quad \text{and} \quad T_x(x_1\otimes\ldots\otimes x_n) = x\otimes x_1\otimes\ldots\otimes x_n\]
for $b\in\B$ and $x, x_1,\ldots, x_n\in X$. The $\rC^*$-algebra $\T_X$ generated by the left action of $\B$ on $\F_X$ and the left-creation operators on $\F_X$ is called the \emph{Toeplitz algebra} of the $\rC^*$-correspondence $X$. The norm-closed subalgebra $\T_X^+$ generated by the left action of $\B$ on $\F_X$ and the left-creation operators is called the \emph{tensor algebra} of $X$.

Let $\B, \C$ be $\rC^*$-algebras and $X$ be a $\rC^*$-correspondence over $\B$. A \emph{representation} of $X$ is a pair $(\rho, t)$ consisting of a non-degenerate $*$-homomorphism $\rho:\B\rightarrow \C$ and a completely contractive linear map $t:X\rightarrow \C$ with the property that
\[\rho(a)t(x)\rho(b) = t(\varphi_X(a)xb), \quad a,b\in \B, x\in X.\]
If, in addition, we have $t(x)^*t(y) = \rho(\langle x, y\rangle)$ for each $x,y\in X$, then we say that $(\rho, t)$ is \emph{rigged} (this is sometimes called \emph{isometric}, for instance in \cite{muhly1998tensor}). In \cites{pimsner1996class, muhly1998tensor}, it is shown that the Toeplitz algebra $\T_X$ is the universal $\rC^*$-algebra generated by the rigged representations of $X$. Given a representation $(\rho,t)$ of $X$ with co-domain $\bB(\H)$, we say that a representation $(\tau,s)$ of $X$ with co-domain $\bB(\K)$ is a \emph{dilation} of $(\rho,t)$ if for every $b\in \B$ and every $\xi \in X$, we have
\[
        P_{\H} \tau(b)|_{\H} = \rho(b), \qquad P_{\H} s(\xi)|_{\H} = t(\xi).
\]
We will say that $(\tau,s)$ is a \emph{minimal} dilation of $(\rho,t)$ if $C^*(\tau,s)\H = \K$.

Given a $\rC^*$-correspondence $X$ over $\B$, we define \emph{Katsura's ideal} $J_X \lhd \B$ by 
\[
    J_X:= \{ \ b\in \B \ | \ \varphi_X(b) \in \K(X) \ \ \text{and} \ \ bc = 0 \ \ \text{for all } c\in \ker \varphi_X \ \}.
\]
Now, given a rigged representation $(\rho, t)$ of $X$, one may define a $*$-homomorphism $\psi_t:\K(X)\rightarrow \C$ by specifying it on rank-one operators by $\psi_t(\theta_{x,y}) = t(x)t(y)^*$ for each $x,y\in X$. A rigged representation $(\rho, t)$ is then said to be \emph{covariant} if $\psi_t(\varphi_X(b)) = \rho(b)$ for each $b\in J_X$. The universal $\rC^*$-algebra generated by rigged covariant representations of $X$ is the \emph{Cuntz-Pimsner algebra} $\O_X$, introduced by Pimsner \cite{pimsner1996class} and refined by Katsura \cite{katsura2004c}. Thus, if $(\hat{\rho},\hat{t})$ is a universal rigged representation (so that the $\rC^*$-algebra generated by the images of $\hat{\rho}$ and $\hat{t}$ is a $*$-isomorphic copy of $\T_X$), we see that $\O_X$ is the quotient of $\T_X$ by the ideal generated by the differences $\psi_{\hat{t}}(\varphi_X(b)) - \hat{\rho}(b)$ for $b\in J_X$. In fact, by \cite[Theorem 3.7]{katsoulis2006tensor}, we know that $\O_X$ can be identified with the $\rC^*$-envelope $\rC^*_e(\T_X^+)$. In particular, the canonical quotient map from $\T_X$ to $\O_X$ is completely isometric on $\T_X^+$. So, the tensor algebra can be thought of as an operator subalgebra in both $\T_X$ and $\O_X$.

We also use the work of Bilich \cite{bilich2025maximality}, which characterizes $*$-representations of $\T_X$ that have the unique extension property with respect to the tensor algebra $\T_X^+$. For a representation $(\rho,t)$ of $X$ into $\mathbb{B}(\mathcal{H})$, consider the contraction defined by
\[
    \widetilde{t} : X \otimes_{\rho} \mathcal{H}\rightarrow\H, \qquad \widetilde{t}(x\otimes h) = t(x)h.
\]
Let $Q\in B(\H)$ be the projection onto $\overline{\spn} \{ t(X) \mathcal{H}\}$. Then, we say that $(\rho,t)$ \emph{fully co-isometric} if $\widetilde{t}\widetilde{t}^* = I$. A rigged representation $(\rho,t)$ is fully co-isometric if and only if $\psi_{t}(\K(X))$ acts non-degenerately on $\H$, which is further equivalent to $Q= I_{\mathcal{H}}$. Let $\widetilde{\varphi}_X : \mathcal{B}^{**} \rightarrow \mathcal{L}(X^{**}) \cong \mathcal{K}(X)^{**}$ be the normal extension of the left action, and denote by $P_X$ the central projection of $\mathcal{B}^{**}$ satisfying $\ker \widetilde{\varphi}_X = (1-P_X) \mathcal{B}^{**}$. If we take $\overline{\rho} : \mathcal{B}^{**} \rightarrow \mathbb{B}(\mathcal{H})$ to be the normal extension of $\rho$, then we always have that $Q \leq \ol{\rho}(P_X)$. Bilich's criterion \cite[Theorem 4.4 and Proposition 2.8]{bilich2025maximality} states that $Q= \overline{\rho}(P_X)$ if and only if $\rho \rtimes t : \mathcal{T}_X \rightarrow \mathbb{B}(\mathcal{H})$ has the unique extension property relative to $\mathcal{T}_X^+$. When $\rho \rtimes t$ arises from a covariant pair $(\rho,t)$, then it is clear that $Q= \overline{\rho}(P_X)$ if and only if $\rho\rtimes t :\mathcal{O}_X \rightarrow \mathbb{B}(\mathcal{H})$ has the unique extension property relative to $q(\mathcal{T}_X^+) \cong \mathcal{T}_X^+$.

Finally, given a locally compact Hausdorff group $G$, a \emph{generalized gauge action} $\alpha:G\acts X$ is an action $\alpha$ of $G$ on $\T_X^+$ such that $\alpha_g(\B) = \B$ and $\alpha_g(X) = X$ for every $g\in G$. Equivalently, a generalized gauge action is the restriction of an action $\alpha$ of $G$ on $\T_X$, or $\O_X$, to $\T_X^+$ such that $\alpha_g(\B) = \B$ and $\alpha_g(X) = X$ for each $g\in G$ with respect to the canonical copies of $\B$ and $X$ in each of these generated $\rC^*$-algebras. By the above description of the Cuntz-Pimsner algebra, we see that $\alpha$ automatically induces an action on $\O_X$, which we will continue to denote by $\alpha$. For a generalized gauge action, one may construct a $\rC^*$-correspondence $X\rtimes_{\alpha, r} G$ over $\B\rtimes_{\alpha, r}G$ (see \cite[Section 2]{bedos2015new} for an equivalent definition) by taking the closures of $\rC_c(G, X)$ and $\rC_c(G, \B)$ considered canonically inside $\T_X\rtimes_{\alpha, r}G$ (which is itself a completion of $\rC_c(G,\T_X)$), where the bimodule actions are induced by multiplication and the $(\B\rtimes_{\alpha, r}G)$-valued inner product is given by $\langle f, g\rangle = f^*g$ for $f,g\in\rC_c(G, X)$.

\section{Reduced crossed products and their C*-envelopes}\label{s:cross-prod-env}

The goal of this section is to prove the commutation of the $\rC^*$-envelope with the reduced crossed product of an operator algebra dynamical system, under a particular separation condition (Theorem \ref{t:red-cross-iso}). We then show that this holds for dynamical systems arising from generalized gauge actions (Proposition \ref{p:cond-S}), which results in a positive answer to the reduced Hao-Ng isomorphism (Theorem \ref{t:hao-ng}).

Let $G$ be a locally compact Hausdorff group and $m$ denote the left Haar measure on $G$. Consider an operator algebra $\A$ with a contractive approximate identity and a point-norm continuous action $\alpha:G\curvearrowright\A$ by completely isometric automorphisms. Throughout, we will refer to the triple $(\A, G, \alpha)$ as a \emph{dynamical system}.

\begin{comment}
We denote by $\sigma : C^*_e(\mathcal{A}) \rightarrow \mathbb{B}(\mathcal{H})$ the largest subrepresentation of the universal representation of $C^*_e(\mathcal{A})$ that has the unique extension property with respect to $\mathcal{A}$. This representation is known to be injective by the Dritschel--McCullough theorem \cite{dritschel2005boundary}. We denote by $z_{\partial}$ the central projection in $C^*_e(\mathcal{A})^{**}$ that satisfies $\ker \sigma^{**} = (I - z_{\partial})C^*_e(\mathcal{A})^{**}$. Throughout, we consider the following separation condition:
\begin{equation} \label{eq:condition-S}
    z_{\partial}C^*_e(\mathcal{A})^* \cap C^*_e(\mathcal{A})^*_c \ \ \text{separates points in} \ \ C^*_e(\mathcal{A})
    \tag{$\mathrm{S}$}.
\end{equation}
\end{comment}

A completely contractive representation $\pi:\A\rightarrow\bB(\H)$ is said to be \emph{$\alpha$-admissible} if there is an action $\widehat{\alpha} : G \curvearrowright \rC^*(\pi(\mathcal{A}))$ such that $\widehat{\alpha}_g \circ \pi = \pi \circ \alpha_g$. When the $\alpha$-admissible representation $\pi$ is completely isometric and considered as an embedding, we will abuse notation and continue to denote the extended action $\widehat{\alpha}$ by $\alpha$. A dynamical system always admits a completely isometric $\alpha$-admissible embedding. For instance, the embedding $\varepsilon:\A\rightarrow\rC^*_e(\A)$ is always $\alpha$-admissible \cite[Lemma 3.3]{katsoulis2019crossed}. A \emph{covariant pair} $(\pi, \mu)$ for $(\A, G, \alpha)$ consists of a non-degenerate representation $\pi:\A\rightarrow\bB(\H)$ and a strongly continuous representation $\mu: G \rightarrow\bB(\H)$ such that
\[
    \pi(\alpha_g(a)) = \mu_g \pi(a)\mu_g^*, \quad g\in G.
\]
Note that whenever $(\pi,\mu)$ is a covariant pair, it follows that $\pi$ is automatically $\alpha$-admissible where the action $\widehat{\alpha}$ of $G$ on $C^*(\pi(A))$ is given by $\widehat{\alpha}(T) = \mu_gT\mu_g^*$. Now, given a non-degenerate representation $\pi : \A \rightarrow \bB(\H)$, we may form the non-degenerate representation $\pi_{\alpha} : \A \rightarrow \bB(\H) \overline{\otimes} L^{\infty}(G) \subseteq \bB(\H \otimes L^2(G))$ by setting $\pi_{\alpha}(a)(g) := (\pi \circ \alpha_g^{-1})(a)$ for $g\in G$. Moreover, note that since $g \mapsto \pi \circ \alpha_g$ is point-norm continuous, each scalar matrix coefficient of $\pi_{\alpha}$ lies in $C_b(G)$. The pair $(\pi_{\alpha},I \otimes \lambda)$ is then readily verified to be covariant. In particular, $\pi_{\alpha}$ is automatically $\alpha$-admissible. Throughout, we will consider the following condition on $(\A, G, \alpha)$:

\begin{enumerate}[label=($\mathrm{S}$),ref=$\mathrm{S}$] 
\item \label{eq:condition-S}
There is an injective $*$-representation $\pi:\rC^*_e(\A)\rightarrow\bB(\H)$ that has the unique extension property with respect to $\A$ and a strongly continuous unitary representation $u:G\rightarrow\bB(\H)$ such that $\pi(\alpha_g(t)) = u_g\pi(t)u_g^*$ for every $g\in G$ and $t\in\rC^*_e(\A)$. 
\end{enumerate}

We note that when $\A$ is hyperrigid or $G$ is discrete, then condition \eqref{eq:condition-S} is automatic. Indeed, when $\A$ is hyperrigid \cite[Theorem 2.1]{arveson2011noncommutative}, any injective covariant pair $(\pi,u)$ of $C^*_e(\A)$ would automatically satisfy that $\pi$ has the unique extension property with respect to $\A$. In which case, the pair $(\pi_{\alpha}, I \otimes \lambda)$ witnesses condition \eqref{eq:condition-S}. When $G$ is discrete, then the invariance principle of the unique extension property (see for instance the discussion after \cite[Proposition 2.4]{dor2018full}) implies that, whenever an injective $*$-representation $\pi : \rC^*_e(\A)\rightarrow\bB(\H)$ has the unique extension property, then $\pi \circ \alpha_g$ has the unique extension property for every $g\in G$. Thus, $\oplus_{g\in G} (\pi \circ \alpha_g^{-1})$ has the unique extension property \cite[Proposition 4.4]{arveson2011noncommutative} and, since the direct sum is unitarily equivalent to $\pi_{\alpha}$, we have that $(\pi_{\alpha}, I \otimes \lambda)$ witnesses condition \eqref{eq:condition-S}.

Given a dynamical system $(\A,G,\alpha)$ and a covariant pair $(\pi,\mu)$, one may form the integrated form $\pi \rtimes \mu$ on $C_c(G,\A)$ given by $[\pi \rtimes \mu](f) = \int_G \pi(f(g))\mu_g dm(g)$, where the latter integral is understood as in \cite[Lemma 1.91]{williams2007book}. Since the reduced crossed product for $(\A, G, \alpha)$ is independent of the choice of a completely isometric $\alpha$-admissible $\pi$ \cite[Corollary 3.16]{katsoulis2019crossed}, we define the reduced crossed product as follows.

\begin{definition}
\emph{Let $(\A, G, \alpha)$ be a dynamical system and $\pi : \mathcal{A} \rightarrow \mathbb{B}(\mathcal{H})$ be a completely isometric $\alpha$-admissible representation. Then, the \emph{reduced crossed product} $\mathcal{A} \rtimes_{\alpha,r} G$ is the closure of the image of $C_c(G,\A)$ under $\pi_{\alpha} \rtimes (I \otimes \lambda)$.}
\end{definition}

The following observation shows that if condition \eqref{eq:condition-S} is witnessed by some representation $\pi$, then the unique extension property for $\pi_{\alpha}$ becomes automatic.

\begin{proposition} \label{p:tmax-action-pres}
    Let $(\A, G, \alpha)$ be a dynamical system that satisfies condition \eqref{eq:condition-S}, and which is witnessed by a $*$-representation $\pi : \rC^*_e(\A) \rightarrow \bB(\H)$. Then, $\pi_\alpha: \rC^*_e(\A)\rightarrow \bB(\H\otimes L^2(G))$ has the unique extension property with respect to $\A$.
\end{proposition}

\begin{proof}
    Let $u : G \rightarrow \bB(\H)$ be a strongly continuous unitary representation such that $\pi(\alpha_g(t)) = u_g\pi(t)u_g^{-1}$. Define the unitary operator $U\in\bB(L^2(G;\H)) \cong \bB(\H) \overline{\otimes} \bB(L^2(G))$ by $U(\xi)(g) = u_g(\xi(g))$. Then, observe that 
    \begin{align*}
         (\Ad_U\circ\pi_\alpha)(t)(\xi)(g) & = (\Ad_{u_g}(\pi(\alpha_{g^{-1}}(t))))(\xi(g))\\
         & = (u_g \pi(\alpha_{g^{-1}}(t))u_g^*)(\xi(g))\\
         & = \pi(t)(\xi(g))\\
         & = (\pi(t)\otimes I)(\xi)(g)
    \end{align*}
    for every $t\in\rC^*_e(\A)$, $\xi\in L^2(G;\H)$, and $g\in G$. That is, we have $\Ad_U \circ \pi_{\alpha} = \pi \otimes \id$. Since $\pi$ has the unique extension property with respect to $\A$, so does $\pi \otimes \id$. Thus, it follows that $\pi_{\alpha}$ has the unique extension property with respect to $\A$.
\end{proof}

In what follows, we will require the use of the \emph{multiplier algebra} of an approximately unital operator algebra $\A$, which is the unital operator algebra of left-right multipliers
\[
    M(\A) = \{ x\in\A^{**} ~:~ x\A\subset\A, \ \A x\subset\A\}.
\]
One may refer to \cite[Chapter 5]{lance1995hilbert} and \cite[Section 2.6]{blecher2004operator} for more, and specifically to \cite[Theorem 2.6.2 and 2.6.7]{blecher2004operator} for a few other equivalent formulations of $M(\A)$. Recall that a net $(x_\lambda)_\lambda$ in $M(\A)$ is said to converge in the strict topology to $x\in M(\A)$ precisely when
\[
    \|x_\lambda a - xa\|, \ \| ax_\lambda-ax\|\rightarrow0, \quad a\in\A.
\]
Now, suppose that $\A$ and $\B$ are operator algebras with contractive approximate identities. We say that a map $\varphi:\A\rightarrow M(\B)$ is \emph{non-degenerate} if, for \emph{some} contractive approximate identity $(e_i)$ of $\A$, the net $(\varphi(e_i))$ forms a contractive approximate identity for $\B$. We will have two situations where a non-degenerate map $\varphi : \A \rightarrow M(\B)$ extends to a unital map $\overline{\varphi} : M(\A) \rightarrow M(\B)$ that is strictly continuous on the unit ball, namely when
\begin{enumerate}
    \item $\A$ is a $\rC^*$-algebra and $\varphi$ is completely positive \cite[Corollary 5.7]{lance1995hilbert}, and;
    \item $\varphi$ is a completely contractive homomorphism \cite[Proposition 2.6.12]{blecher2004operator}.
\end{enumerate}
For instance, when $\A \subseteq M(\B)$ is a non-degenerate containment of operator algebras, from \cite[Proposition 2.6.12]{blecher2004operator}, we may deduce the unital containment $M(\A) \subseteq M(\B)$. We note that our definition of non-degeneracy (as opposed to the definition of multiplier non-degenerate morphism in \cite[2.6.11]{blecher2004operator}) is sufficient for the proof of \cite[Proposition 2.6.12]{blecher2004operator} to go through verbatim.

When $(\A,G,\alpha)$ is a dynamical system, since $\A$ has a contractive approximate identity $(e_{\alpha})_{\alpha}$, by \cite[Lemma 3.5]{katsoulis2019crossed} there is a contractive approximate identity $(f_{\lambda})_{\lambda}$ of $C_c(G)$ (in the $L^1$ norm) so that the image of $(e_{\alpha} \cdot f_{\lambda})_{\alpha,\lambda}$ under $\pi_{\alpha} \rtimes (I \otimes \lambda)$ is an approximate identity for $\A \rtimes_{\alpha,r} G$. 

We now prove that the reduced crossed product of a locally compact group action on an operator algebra commutes with the $\rC^*$-envelope under condition \eqref{eq:condition-S}. This improves on \cite[Theorem 2.5]{katsoulis2017c} and \cite[Theorem 3.6]{katsoulis2021non} by identifying a broader sufficient condition under which the conclusion holds, while removing the assumption of the contractive approximate identity being \emph{self-adjoint}.

\begin{theorem}\label{t:red-cross-iso}
    Let $(\A, G, \alpha)$ be a dynamical system satisfying condition \eqref{eq:condition-S}. Then, $\rC^*_e(\A) \rtimes_{\alpha,r} G \cong \rC^*_e(\A \rtimes_{\alpha,r} G)$ via the canonical, generator preserving map.
\end{theorem}

\begin{proof}
    By assumption, there is an isometric $*$-representation $\pi: \rC^*_e(\A)\rightarrow\bB(\H)$ that has the unique extension property with respect to $\A$ and such that $\alpha$ is implemented by a unitary representation $u:G\rightarrow\bB(\H)$. Without loss of generality, we may assume that $\pi$ is non-degenerate. Take $\psi:=\pi_{\alpha} \rtimes (I \otimes \lambda)$. Since $\psi$ is completely isometric on $\A\rtimes_{\alpha,r}G$, it will suffice to show that $\psi$ has the unique extension property with respect to $\A\rtimes_{\alpha,r}G$. To that end, let $\varphi: \rC^*_e(\A)\rtimes_{\alpha, r} G\rightarrow  \bB(\H \otimes L^2(G))$ be a completely contractive completely positive map such that $\varphi|_{\A\rtimes_{\alpha,r}G} = \psi|_{\A\rtimes_{\alpha,r}G}$. We will show that $\varphi = \psi$.
    
    Since $\A$ has a contractive approximate identity $(e_{\alpha})_{\alpha}$, we see that $\A\rtimes_{\alpha, r}G$ has a contractive approximate identity, which is the image of $(e_{\alpha}\cdot f_\lambda )_{\alpha,\lambda}$ under $\psi = \pi_{\alpha} \rtimes (I \otimes \lambda)$ with $(f_\lambda)_\lambda\subseteq \rC_c(G)$ being a contractive approximate identity in $L^1(G)$. Thus, this contractive approximate identity in $\A \rtimes_{\alpha,r} G$ is also a contractive approximate identity for $\rC^*_e(\A)\rtimes_{\alpha, r}G$. Since $\varphi(e_{\alpha}\cdot f_\lambda) = \psi(e_{\alpha}\cdot f_\lambda) \rightarrow I_{\H \otimes L^2(G)}$ in the strong* topology, by \cite[Corollary 5.7]{lance1995hilbert} there is a unique unital completely positive map
    \[
        \Phi: M(\rC^*_e(\A)\rtimes_{\alpha, r}G)\rightarrow \bB(\H \otimes L^2(G))
    \]
    which is strictly continuous on the unit ball and extends $\varphi$ on $\rC^*_e(\A)\rtimes_{\alpha, r} G$. By \cite[Proposition 2.6.12]{blecher2004operator}, there is a unital $*$-homomorphism $\Psi :M(\rC^*_e(\A)\rtimes_{\alpha, r}G)\rightarrow \bB(\H \otimes L^2(G)) $ which is strictly continuous on the unit ball and extends $\psi$ on $\rC^*_e(\A)\rtimes_{\alpha, r}G$. Now, since  $\rC^*_r(G)$, $\rC^*_e(\A)$, and $\A\rtimes_{\alpha, r}G$ have contractive approximate identities that are also approximate identities for $\rC^*_e(\A)\rtimes_{\alpha, r}G$, we also have that $M(\rC^*_e(\A)) \subseteq M(\rC^*_e(\A)\rtimes_{\alpha, r}G)$ and that $M(\rC^*_r(G))\subseteq M(\A\rtimes_{\alpha, r}G) \subseteq M(\rC^*_e(\A)\rtimes_{\alpha, r}G)$. Since $\phi$ coincides with $\psi$ on $\A\rtimes_{\alpha, r}G$, we must have that $\Phi$ coincides with $\Psi$ on ${M(\A\rtimes_{\alpha, r} G)}$ by strict continuity of $\Phi$ and $\Psi$ on the unit ball. In particular, since $M(\rC^*_r(G)) \subseteq M(\A \rtimes_{\alpha,r} G)$, we see that $\Phi$ must coincide with $\Psi$ on $M(\rC^*_r(G))$ pointwise, and so $\Phi$ must be multiplicative there.

    We claim that $\Phi$ must coincide with $\Psi$ on $\rC^*_e(\A)$. To this end, consider the completely contractive completely positive map $\theta : = \Phi|_{\rC^*_e(\A)}$, which is an extension of $\Psi|_{\A} = \pi_{\alpha}|_{\A}$. By Proposition \ref{p:tmax-action-pres}, we have that $\pi_{\alpha}$ has the unique extension property with respect to $\A$ and thus, $\theta = \pi_{\alpha}$ on $\rC^*_e(\A)$. In turn, since $\Phi$ is strictly continuous on the unit ball, we see that $\Phi$ coincides with $\Psi$ on $M(\rC^*_e(\A))$, so that $\Phi$ is multiplicative on $M(\rC^*_e(\A))$. Since both $M(\rC^*_e(\A))$ and $M(\rC^*_r(G))$ belong to the multiplicative domain of $\Phi$, it follows that $\Phi$ coincides with $\Psi$ on $\rC^*_e(\A) \rtimes_{\alpha,r} G$. So, $\varphi=\psi$ and we see that $\psi$ has the unique extension property relative to $\A \rtimes_{\alpha,r} G$.
\end{proof}

We now establish condition \eqref{eq:condition-S} for dynamical systems that arise from a generalized gauge action on a $\rC^*$-correspondence. For the notation used in the proof below, we refer the reader to Subsection \ref{ss:correspondence}.

\begin{proposition} \label{p:cond-S}
    Let $X$ be a non-degenerate C*-correspondence over a $\rC^*$-algebra $\B$, and let $\alpha : G \acts X$ be a generalized gauge action. Then, the tensor algebra dynamical system $(\mathcal{T}_X^+, G, \alpha)$ satisfies condition \eqref{eq:condition-S}.
\end{proposition}

\begin{proof}
     By \cite[Theorem 3.7]{katsoulis2006tensor}, we have that $\O_X \cong \rC^*_e(\T_X^+)$ and so, we construct an injective $*$-representation of $\O_X$ that has the unique extension property with respect to $\T_X^+$ and such that the action $\alpha:G\acts\O_X$ is unitarily implemented. For this, we apply the tail-adding technique of Muhly and Tomforde \cite[Section 4]{muhly2004tailadding}.
     
     Let $\I = \ker\varphi_X$ and consider the $\rC^*$-correspondence $Y:= X \oplus c_0(\bN,\I)$ over the $\rC^*$-algebra $\C := \B \oplus c_0(\bN,I)$. By \cite[Lemma 4.2]{muhly2004tailadding}, we have that $Y$ is an injective, non-degenerate $\rC^*$-correspondence over $\C$. Then, let $\pi : \C \rightarrow \bB(\H_\pi)$ be an injective non-degenerate $*$-representation and, on $\H_0:= Y \otimes_\pi \H_\pi$, define a $*$-representation $\rho_0(c):= \varphi_Y(c) \otimes 1$ for $c\in \mathcal{C}$ and a map $t_0(y) = 0$ for $y\in Y$. The pair $(\rho_0,t_0)$ then forms a representation of $Y$ where $\rho_0$ is injective. Since $\mathcal{K}(Y)$ acts non-degenerately on $Y$ and $Y$ is strict (see the definition before \cite[Lemma 2.3]{muhly1998tensor}), we see that $\rho_0$ is strict to strong* continuous. Then, by \cite[Theorem 5.10]{muhly1998tensor}, there exists a fully co-isometric representation $(\rho_1,t_1)$ on a Hilbert space $\H_1$ such that $\H_0$ is an invariant subspace of $\H_1$ for $(\rho_1,t_1)$. Moreover, by \cite[Theorem 3.3]{muhly1998tensor}, there exists a minimal \emph{rigged} dilation $(\rho, t)$ of $(\rho_1,t_1)$ on a Hilbert space $\H_2$ such that $\H_1$ is a co-invariant subspace of $\H_2$ for $(\rho,t)$. Since $(\rho_1,t_1)$ is fully co-isometric, it follows from minimality of the dilation that $(\rho,t)$ is also fully co-isometric. Since fully co-isometric rigged representations are Cuntz-Pimsner covariant, we then obtain an induced $*$-representation $\rho \rtimes t : \O_Y \rightarrow \bB(\H_2)$ and, as $\rho_0$ is injective, so is $\rho$.

    Since $\I$ is invariant under the action $\alpha : G \acts X$, we then get an induced coordinate-wise action $\beta : G \acts Y$. We continue to denote by $\alpha$ and $\beta$ the induced actions on $\O_X$, and $\O_Y$, respectively. Since $\beta$ commutes with the gauge action $\gamma : \bT \acts \O_Y$, there is a well-defined action $\delta : G\times \bT \acts \O_Y$ given by $\delta_{g,z} = \beta_g \gamma_z$. Consider the covariant pair $((\rho\rtimes t)_{\delta}, I \otimes \lambda)$ on the Hilbert space $\H_2 \otimes L^2(G\times \bT)$. Let $(\widehat{\rho},\widehat{t})$ be a universal rigged covariant representation of $Y$. Since $(\rho,t)$ is fully co-isometric, it follows that the $\delta$-invariant subalgebra $\psi_{\widehat{t}}(\K(Y)) \subseteq \O_Y$ acts non-degenerately under $\rho\rtimes t$ and so, it also acts non-degenerately under $(\rho\rtimes t)_{\delta}$. As the gauge action by $\gamma_z$ is implemented in the range by adjoining by $I\otimes \lambda_{(e,z)}$, it follows that $(\rho\rtimes t)_{\delta}$ is gauge equivariant. Since $\rho$ is injective, it then follows that $(\rho\rtimes t)_{\delta}$ is injective by the gauge-invariant uniqueness theorem \cite[Theorem 6.4]{katsura2004c}. Letting $(\tau,s)$ be the rigged representation of $Y$ corresponding to $(\rho\rtimes t)_{\delta}$, we then have that $\tau$ is injective and, as $\psi_{\widehat{t}}(\K(Y))$ acts non-degenerately under $(\rho\rtimes t)_{\delta}$, that $(\tau,s)$ is fully co-isometric.

    By \cite[Theorem 4.3(c) and Lemma 4.11]{muhly2004tailadding}, there exists a projection $p \in M(\O_Y)$ satisfying $\O_X = p\O_Y p$. Since $1-p$ is a strict limit of an approximate identity in the tail ideal $c_0(\bN,\I)$, it follows that $p$ is invariant under $\delta$. Thus, if we compress $(\tau,s)$ by $\overline{\tau}(p)$ to a subspace $\L$, then we get a covariant rigged pair $(\sigma,r)$ for $X$. Then, $\sigma \rtimes r$ is an injective $*$-representation of $\O_X$, and there is a strongly continuous unitary representation $V:G\rightarrow \bB(\L)$, given by $V_g := (I\otimes \lambda_{(g,1)})|_{\L}$, that has the property that $(\sigma \rtimes r)(\alpha_g(a)) = V_g (\sigma\rtimes r)(a)V_g^*$ for every $g\in G$ and $a\in\O_X$.

    We show that $\sigma\rtimes r$ has the unique extension property with respect to $\T_X^+$. Indeed, for $\xi,\eta \in X$ and $i_k,j_k \in \I$, we see \cite[Lemma 4.7 \& Lemma 4.8(3)]{muhly2004tailadding} gives
    \[
        s(\xi, (i_1,i_2,...))s(\eta, (j_1,j_2,...))^* = s(\xi,\vec{0})s(\eta,\vec{0})^* + \tau(i_1j_1^*, (i_2j_2^*,...)),
    \]
    so that \cite[Lemma 4.11]{muhly2004tailadding} then yields
    \begin{align*}
        \overline{\tau}(p)s(\xi, (i_1,i_2,...))s(\eta, (j_1,j_2,...))^*\overline{\tau}(p) &  = \overline{\tau}(p)[s(\xi,\vec{0})s(\eta,\vec{0})^* + \tau (i_1j_1^*,\vec{0}))]\overline{\tau}(p)\\ 
        & = \sigma(\langle \xi,\eta \rangle +i_1j_1^*)\\
        & = r(\xi)r(\eta)^* + \sigma(i_1j_1^*).
    \end{align*}
    Since $(\tau,s)$ is fully co-isometric, we get that $\L = \overline{\spn}\{r(X) \L \} \oplus \overline{\spn} \{ \sigma(\I) \L \}$. Letting $Q$ denote the projection onto $\overline{\spn}\{r(X) \L \}$, we have now shown that $Q = 1 -\overline{\sigma}(z_\I)$ where $z_\I \in \B^{**}$ is the central projection satisfying $\I^{**} = z_\I\B^{**}$. Since $\I = \ker \varphi_X$, we have that $P_X \leq 1- z_\I$ and so, $Q \leq \overline{\sigma}(P_X)\leq 1 - \overline{\sigma}(z_\I) = Q$. Therefore, by \cite[Theorem 4.4 and Proposition 2.8]{bilich2025maximality}, we have that $\sigma \rtimes r$ has the unique extension property with respect to $\T_X^+$.
\end{proof}

We are now at the stage where we can settle the Hao-Ng isomorphism problem for reduced crossed products. 

\begin{theorem}\label{t:hao-ng}
    Let $G$ be a locally compact Hausdorff group, and let $X$ be a (non-degenerate) $\rC^*$-correspondence over a $\rC^*$-algebra $\B$. Suppose $\alpha:G\acts X$ is a generalized gauge action. Then, $\O_X\rtimes_{\alpha, r} G\cong \O_{X\rtimes_{\alpha, r} G}$.
\end{theorem}

\begin{proof}
    By Proposition \ref{p:cond-S} the dynamical system $(\mathcal{T}_X^+, G, \alpha)$ satisfies condition \eqref{eq:condition-S}. Thus, by \cite[Theorem 3.7]{katsoulis2006tensor}, Theorem \ref{t:red-cross-iso}, and \cite[Theorem 3.7]{katsoulis2021non}, we have that
    \[
    \O_X \rtimes_{\alpha, r} G \cong  \rC^*_e(\T_X^+)\rtimes_{\alpha, r}G \cong \rC^*_e(\T_X^+ \rtimes_{\alpha, r}G) \cong \O_{X\rtimes_{\alpha, r} G},
    \]
    via generator preserving isomorphisms.
\end{proof}

\vspace{6pt}

\textbf{Acknowledgments.}
The authors are grateful to Boyu Li for several fruitful discussions held at New Mexico State University and electronically over Zoom, and to Boris Bilich and Jamie Gabe for suggestions and comments on a previous version of the paper. The authors are grateful to Rapha\"el Clou\^atre for pointing out a missing argument in a preliminary version of the paper. The authors also thank Yuhei Suzuki for pointing out, on September 25th, 2026, an error in a previous version of this paper, which eventually led us to isolate condition \eqref{eq:condition-S}.

\bibliographystyle{plain}
\bibliography{bbl}

\end{document}